\documentclass[12pt]{article}
\usepackage[latin1]{inputenc}
\usepackage[british]{babel}
\usepackage{lmodern}
\usepackage[T1]{fontenc}
\usepackage[paper=a4paper, left=28mm, right=25mm, top=29mm, bottom=25mm]{geometry}
\usepackage{latexsym,amsfonts,amsmath,graphics}
\usepackage{epsfig}
\usepackage{enumitem}
\usepackage{amssymb,mathrsfs}
\usepackage{verbatim}

\newtheorem{theorem}{Theorem}
\newtheorem{lemma}{Lemma}
\newtheorem{corollary}{Corollary}

\newtheorem{remark}{Remark}
\newenvironment{proof}{\begin{trivlist} \item[\hskip\labelsep{\it Proof.}]}{$\hfill\Box$\end{trivlist}}

\newcommand{\rd}{\,\mathrm{d}}

\newcommand{\bsone}{\boldsymbol{1}}

\newcommand{\RR}{\mathbb{R}}

\newcommand{\NN}{\mathbb{N}}

\newcommand{\cV}{\mathcal{V}}

\newcommand{\sym}{{\rm sym}}
\newcommand{\cH}{\mathcal{V}}

\newcommand{\cS}{\mathcal{S}}

\allowdisplaybreaks

\title{$L_p$-discrepancy of the symmetrized van der Corput sequence}
\author{Ralph Kritzinger and Friedrich Pillichshammer\thanks{The authors are supported by the Austrian Science Fund (FWF): Project F5509-N26, which is a part of the Special Research Program "Quasi-Monte Carlo Methods: Theory and Applications".}}
\date{}

\begin{document}

\maketitle

\begin{abstract}
It is well known that the $L_p$-discrepancy for $p \in [1,\infty]$ of the van der Corput sequence is of exact order of magnitude $O((\log N)/N)$. This however is for $p \in (1,\infty)$ not best possible with respect to the lower bounds according to Roth and Proinov. For the case $p=2$ it is well known that the symmetrization trick due to Davenport leads to the optimal $L_2$-discrepancy rate $O(\sqrt{\log N}/N)$ for the symmetrized van der Corput sequence. In this note we show that this result holds for all $p \in (1,\infty)$. The proof is based on an estimate of the Haar coefficients of the corresponding local discrepancy and on the use of the Littlewood-Paley inequality.
\end{abstract}

\centerline{\begin{minipage}[hc]{130mm}{
}
\end{minipage}}

 \allowdisplaybreaks
\section{Introduction and Statement of the Result} 

For an infinite sequence $\cS =(x_n)_{n \ge 0}$ of points in $[0,1)$
the {\it local discrepancy} of its first $N$ elements is defined as
$$D_N(\cS,t)=\frac{1}{N}\sum_{n=0}^{N-1}\bsone_{\left[0,t\right)}(x_n)- t,$$ where, throughout this paper,
$\bsone_{I}(x)$ denotes the indicator function of the interval $I\subseteq \left[0,1\right]$. The {\it $L_p$-discrepancy} for $p \in [1,\infty]$ of $\cS$ is defined as the $L_p$-norm of the local discrepancy, thus, for $p \in (1,\infty)$
\begin{eqnarray*}
L_{p,N}(\cS):=\| D_N(\cS,\cdot)\|_{L_p}= \left(\int_{0}^{1} |D_N(\cS,t)|^p \rd t\right)^{1/p}.
\end{eqnarray*}
For $p=\infty$ we have $$L_{\infty,N}(\cS):=\| D_N(\cS,\cdot)\|_{L_{\infty}}= \sup_{t \in [0,1]} |D_N(\cS,t)|.$$
The $L_p$-discrepancy is a quantitative measure for the irregularity of distribution of a sequence modulo one, see, e.g., \cite{DT97,kuinie}. It is also related to the worst-case integration error of a quasi-Monte Carlo rule, see, e.g., \cite{DP10,LP14}.

It is well known that for every $p\in(1,\infty]$ there exists a positive number $c_p$ with the property that for every sequence $\cS$ in $[0,1)$ we have 
\begin{equation}\label{lowschmid}
L_{\infty,N}(\cS) \ge c_{\infty} \frac{\log N}{N} \ \ \ \mbox{ for infinitely many $N \in \NN$}
\end{equation}
 and, for $p \in (1,\infty)$, 
\begin{equation}\label{lowproinov}
L_{p,N}(\cS) \ge c_p \frac{\sqrt{\log N}}{N} \ \ \ \mbox{ for infinitely many $N \in \NN$},
\end{equation}
where $\log$ denotes the natural logarithm and where $\NN$ denotes the set of positive integers $\{1,2,3,\ldots\}$. The result for $p=\infty$ was shown by Schmidt~\cite{Schm72distrib} and the result for $p \in (1,\infty)$ was shown by Proinov~\cite{pro86} based on results of Roth~\cite{Roth} and Schmidt~\cite{schX}. Both lower bounds \eqref{lowschmid} and \eqref{lowproinov} are optimal in the order of magnitude in $N$.\\

A prototype of a sequence with low discrepancy is the {\it van der Corput sequence} (in base 2). Let $\varphi(n)$ denote the {\it radical inverse} of $n\in \NN_0$ in base $2$ (where $\NN_0=\NN \cup \{0\}$) which is defined as $\varphi(n):=\sum_{i=0}^{k}n_i2^{-i-1}$ whenever $n$ has binary expansion $n=\sum_{i=0}^{k}n_i2^i$, where $n_i\in \{0,1\}$ for all $i \in \{0,\dots,k\}$. The van der Corput sequence (in base 2) is the sequence $\cV=(y_n)_{n \ge 0}$ where $y_n=\varphi(n)$ for $n \in \NN_0$.

For the van der Corput sequence it is known that (see, e.g, \cite{befa,hab1966}) $$\limsup_{N \rightarrow \infty} \frac{N L_{\infty,N}(\cV)}{\log N}=\frac{1}{3 \log 2}$$ and hence $L_{\infty,N}(\cV)$ is of order of magnitude $O((\log N)/N)$ which is best possible in $N$ according to \eqref{lowschmid}. However, for $p \in (1,\infty)$ it is known that (see, e.g., \cite{chafa,proat} for $p=2$ and \cite{pil04} for general $p$) $$\limsup_{N \rightarrow \infty} \frac{N L_{p,N}(\cV)}{\log N}=\frac{1}{6 \log 2}.$$ This means that $L_{p,N}(\cV)$ for $p \in (1,\infty)$ is only of order of magnitude $O((\log N)/N)$ which is {\it not} best possible in $N$ if we compare with \eqref{lowproinov}.

One way out to overcome this defect of the van der Corput sequence is based on symmetrization which was initially introduced by Davenport for $(n\alpha)$-sequences (see \cite[Theorem~1.75]{DT97}). This method is also known as {\it Davenport's reflection principle}.  

We define the {\it symmetrized van der Corput sequence} (in base 2) $\cH^{\sym}=(z_n)_{n\geq 0}$ as
\[z_n = \begin{cases} \varphi(m) &\mbox{if } n=2m, \\ 
                       1-\varphi(m) & \mbox{if } n=2m+1. \end{cases} 
\]

Then it is known, see e.g. \cite{chafa,fau90,lp,pro1988a}, that the $L_2$-discrepancy of the symmetrized van der Corput sequence is of optimal order of magnitude in $N$ compared to the lower bound in \eqref{lowproinov}, i.e., 
\begin{equation}\label{estsymvdcL2}
L_{2,N}(\cV^{\sym}) \ll \frac{\sqrt{\log N}}{N}.
\end{equation}
Here and throughout the paper, for functions $f,g:\NN \rightarrow \RR^+$, we write $g(N) \ll f(N)$, 
if there exists a $C>0$ such that $g(N) \le C f(N)$ for all $N \in \NN$, $N \ge 2$. If we would like to stress that $C$ depends on some parameter, say $p$, this will be indicated by writing $\ll_p$. 

It is the aim of this paper to show that the estimate \eqref{estsymvdcL2} holds for all $p \in (1,\infty)$. We show:

\begin{theorem}\label{thm1}
For every $p \in (1,\infty)$ we have
    \[ L_{p,N}(\cH^{\sym})\ll_p \frac{\sqrt{\log{N}}}{N}.\]
\end{theorem}

The proof of this result is based on the Haar function system (in base $2$) and will be given in Section~\ref{secproofthm}. First we collect some auxiliary results in the following section. 

\section{Auxiliary Results}

In order to estimate the $L_p$-discrepancy of $\cH^{\sym}$ we use the one-dimensional Haar system. Haar functions are a useful and often applied tool in discrepancy theory, see e.g. \cite{DHP,hin2010,Mar2013,Mar2013a,Mar2013b}.

To begin with, a {\it dyadic interval} of length $2^{-j}$, $j\in {\mathbb N}_0,$ in $[0,1)$ is an interval of the form 
$$ I=I_{j,m}:=\left[\frac{m}{2^j},\frac{m+1}{2^j}\right) \ \ \mbox{for } \  m=0,1,\ldots,2^j-1.$$ We also define $I_{-1,0}=[0,1)$.
The left and right half of $I=I_{j,m}$ are the dyadic intervals $I^+ = I_{j,m}^+ =I_{j+1,2m}$ and $I^- = I_{j,m}^- =I_{j+1,2m+1}$, respectively. The {\it Haar function} $h_I = h_{j,m}$ with support $I$ 
is the function on $[0,1)$ which is  $+1$ on the left half of $I$, $-1$ on the right half of $I$ and 0 outside of $I$. The $L_\infty$-normalized {\it Haar system} consists of
all Haar functions $h_{j,m}$ with $j\in{\mathbb N}_0$ and  $m=0,1,\ldots,2^j-1$ together with the indicator function $h_{-1,0}$ of $[0,1)$.
Normalized in $L_2([0,1))$ we obtain the {\it orthonormal Haar basis} of $L_2([0,1))$. 

The {\it Haar coefficients} of a function $f \in L_p([0,1))$ are defined as $$\mu_{j,m}(f):=\langle f,h_{j,m} \rangle=\int_{0}^{1} f(t)h_{j,m}(t)\rd t\ \ \ \mbox{for $j\in \NN_{-1}$ and $m \in \mathbb{D}_j$},$$ where here and later on we use the abbreviations $\NN_{-1}:=\NN_{0}\cup\{-1\}$, $\mathbb{D}_{j}:=\{0,1,\dots,2^j-1\}$ for $j\in\NN_0$ and $\mathbb{D}_{-1}:=\{0\}$.

In the following, we will compute the {\it Haar coefficients} of the local discrepancy of $\cH^{\sym}$, i.e.,
\[ \mu_{j,m}(D_N(\cH^{\sym},\cdot))=\langle D_N(\cH^{\sym},\cdot),h_{j,m} \rangle=\int_{0}^{1}D_N(\cH^{\sym},t)h_{j,m}(t)\rd t. \]

Preceeding the computation of the Haar coefficients, we collect some properties of the radical inverse function $\varphi(n)$ which
we will need in the proof of the essential Lemma~\ref{Theo1}.

\begin{lemma} \label{Phi}
   The following relations hold for the radical inverse function $\varphi$:
   \begin{enumerate}
       \item  $\varphi(2^js)=\frac{1}{2^j}\varphi(s)$ for all $j,s \in \NN_0$,
       \item \label{pt2} $\varphi(2^j\varphi(m))=\frac{m}{2^j}$ for all $j \in \NN_0$ and $m\in \{0,\dots,2^j-1\}$,
       \item \label{pt3} $\varphi(n) \in I_{j,m}$ if and only if $n=2^j\varphi(m)+2^js$ for some $s \in \NN_0$,
       \item $\frac{A}{2}\leq \sum_{s=0}^{A}\left|1-2\varphi(s)\right|\leq \frac{A}{2}+1$ for all $A\in \NN_0$.
   \end{enumerate}
\end{lemma}

\begin{proof} 
    \begin{enumerate}
       \item Let $s=\sum_{i=0}^{k}s_i2^i$, where $s_i \in \{0,1\}$ for all $i \in \{0,\dots,k\}$. Then $\varphi(2^js)=\varphi(\sum_{i=0}^{k}s_i2^{i+j})=\sum_{i=0}^{k}s_i2^{-i-j-1}=2^{-j}\sum_{i=0}^{k}s_i2^{-i-1}=\frac{1}{2^{j}}\varphi(s).$
       \item Since $0 \leq m \leq 2^j-1$, $m$ has a binary representation of the form $m=\sum_{i=0}^{j-1}m_i2^i$, where $m_i \in \{0,1\}$ for all $i \in \{0,\dots,j-1\}$. Then $2^j\varphi(m)=\sum_{i=0}^{j-1}m_i2^{j-i-1}$
             and therefore $\varphi(2^j\varphi(m))=\sum_{i=0}^{j-1}m_i2^{-(j-i-1)-1}=2^{-j}\sum_{i=0}^{j-1}m_i2^{i}=m2^{-j}$.
       \item We write $n$ in the form $n=\tilde n +2^js$, where $\tilde n\in \{0,1,\dots, 2^j-1\}$ and $s\in \NN_0$. Then $\varphi(\tilde n)=n'2^{-j}$ for some $n'\in \{0,1,\dots,2^j-1\}$ as one can verify easily. We have $\varphi(n)=\frac{n'}{2^j}+\frac{\varphi(s)}{2^j}$. We see that $\varphi(n)\in I_{j,m}$ is true if and only if $n'=m$. But from Point~\ref{pt2}. we know $\varphi^{-1}(m2^{-j})=2^j\varphi(m)$ and the proof of Point~\ref{pt3}. is done.       
       \item To begin with, we verify the relation $|1-2\varphi(2n)|+|1-2\varphi(2n+1)|=1$ for all $n\in \NN_0$. Therefore we observe that $\varphi(2n)\leq \frac{1}{2}$
             for all $n\in \NN_0$. We also have $\varphi(2n+1)=\varphi(2n)+\varphi(1)=\varphi(2n)+\frac{1}{2}$, hence
             \begin{align*}
                |1-2\varphi(2n)|+|1-2\varphi(2n+1)|=&1-2\varphi(2n)+\left|1-2\left(\varphi(2n)+\frac{1}{2}\right)\right| \\
                                             =&1-2\varphi(2n)+|-2\varphi(2n)|=1.
             \end{align*}
             This leads to 
             \[ \sum_{s=0}^{A}\left|1-2\varphi(s)\right|=\sum_{n=0}^{(A-1)/2}\left\{\left|1-2\varphi(2n)\right|+\left|1-2\varphi(2n+1)\right|\right\}=\sum_{n=0}^{(A-1)/2}1=\frac{A+1}{2} \]
             for odd $A$ and
             \begin{align*} \sum_{s=0}^{A}\left|1-2\varphi(s)\right|=&\sum_{n=0}^{(A-2)/2}\left\{\left|1-2\varphi(2n)\right|+\left|1-2\varphi(2n+1)\right|\right\}+|1-2\varphi(A)| \\ =&\sum_{n=0}^{(A-2)/2}1+\left|1-2\varphi(A)\right|=\frac{A}{2}+\left|1-2\varphi(A)\right| \end{align*} for even $A$. 

             Hence in both cases we have $\frac{A}{2}\leq \sum_{s=0}^{A}\left|1-2\varphi(s)\right|\leq \frac{A}{2}+1$. 
   \end{enumerate}

\end{proof}

\begin{lemma} \label{erster}
   The Haar coefficient $\mu_{-1,0}$ of the local discrepancy $D_N(\cH^{\sym},\cdot)$ satisfies
   \[|\mu_{-1,0}|= \begin{cases} 0 &\mbox{if } N=2M, \\ 
                       \left|\frac{1}{2N}-\frac{\varphi(M)}{N}\right|\leq \frac{1}{2N}\ & \mbox{if } N=2M+1. \end{cases} 
\]
\end{lemma}

\begin{proof}
  We have
  \begin{align*}
    \mu_{-1,0}=& \int_{0}^{1}D_N(\cH^{\sym},t)\rd t =\int_{0}^{1}\left(\frac{1}{N}\sum_{n=0}^{N-1}\bsone_{\left[0,t\right)}(x_n)-t\right)\rd t \\ =&\frac{1}{N}\sum_{n=0}^{N-1}\int_{0}^{1}\bsone_{\left[0,t\right)}(x_n)\rd t-\frac{1}{2}
    =\frac{1}{N}\sum_{n=0}^{N-1}\left(\int_{x_n}^{1}\bsone_{\left[0,t\right)}(x_n)\rd t\right)-\frac{1}{2} \\
    =&\frac{1}{N}\sum_{n=0}^{N-1}(1-x_n)-\frac{1}{2}=\frac{1}{2}-\frac{1}{N}\sum_{n=0}^{N-1}x_n.
  \end{align*}
  We therefore have to investigate the sum $\sum_{n=0}^{N-1}x_n$. If $N=2M$, then we have
  \[\sum_{n=0}^{2M-1}x_n=\sum_{n=0}^{M-1}\varphi(m)+\sum_{m=0}^{M-1}(1-\varphi(m))=M=\frac{N}{2}.\]
  For $N=2M+1$, we find
  \[\sum_{n=0}^{2M}x_n=\sum_{n=0}^{2M-1}x_n+\varphi(M)=M+\varphi(M)=\frac{N-1}{2}+\varphi\left(M\right).\]
  This leads to the desired result.
\end{proof}

\begin{remark}\rm
Lemma~\ref{erster} is a crucial fact. It shows how the symmetrization trick keeps the Haar coefficient $\mu_{-1,0}$ small enough in order to achive the optimal $L_p$-discrepancy rate. Let us compare this with the behavior of the Haar coefficient $\mu_{-1,0}^{N,\varphi}$ of the local discrepancy of the first $N$ terms of the {\it usual (not symmetrized)} van der Corput sequence:  for $2^m \le N < 2^{m+1}$ it is known that $$\mu_{-1,0}^{N,\varphi}=\frac{1}{2 N}\left(1+\sum_{r=0}^{m-1}\left\|\frac{N}{2^{r+1}}\right\|\right),$$ where $\|x\|$ denotes the distance of $x$ to the nearest integer, see, e.g., \cite[Proposition~1]{DLP05}. Further we know from \cite[Theorem~3]{LP2003} that $$\max_{2^m \le N < 2^{m+1}} \sum_{r=0}^{m-1}\left\|\frac{N}{2^{r+1}}\right\|=\frac{m}{3}+\frac{1}{9}-(-1)^m \frac{1}{9\cdot 2^m}.$$ Hence it follows that there exists some constant $c>0$ such that $$\mu_{-1,0}^{N,\varphi} \ge c \frac{\log N}{N}  \ \ \ \mbox{ for infinitely many $N \in \NN$}.$$ Therefore, for the usual van der Corput sequence, already the first (and only the first, c.f. Lemma~\ref{Theo1}) Haar coefficient of the local discrepancy has the ``bad'' order of magnitude $(\log N)/N$. 
\end{remark}

In the following, we write $D_{N}^{\sym}(t):=D_N(\cH^{\sym},t)$ and denote by $D_N^{\varphi}(t)$ the local discrepancy of the first $N$ elements of the sequence $(\varphi(n))_{n\geq 0}$ and
analogously $D_N^{1-\varphi}(t)$ the local discrepancy of the first $N$ elements of the sequence $(1-\varphi(n))_{n\geq 0}$.
Let $\mu_{j,m}^{N,\sym}$, $\mu_{j,m}^{N,\varphi}$ and $\mu_{j,m}^{N,1-\varphi}$ be the Haar coefficients of these three functions.\begin{lemma} \label{Diskr}
    For $N=2M$ we have
    \[ D_{2M}^{\sym}(t)=\frac{1}{2}\left( D_{M}^{\varphi}(t)+D_{M}^{1-\varphi}(t) \right) \]
    and for $N=2M+1$
    \[ D_{2M+1}^{\sym}(t)=\frac{1}{2M+1}\left( (M+1)D_{M+1}^{\varphi}(t)+MD_{M}^{1-\varphi}(t) \right). \]
\end{lemma}

\begin{proof} For $N=2M$ we have
   \begin{align*} D_{2M}^{\sym}(t)=&\frac{1}{2M}\left(\sum_{n=0}^{M-1}\bsone_{\left[0,t\right)}(\varphi(n))+\sum_{n=0}^{M-1}\bsone_{\left[0,t\right)}(1-\varphi(n))\right)-t \\
     =&\frac{1}{2}\left(\frac{1}{M}\sum_{n=0}^{M-1}\bsone_{\left[0,t\right)}(\varphi(n))-t+\frac{1}{M}\sum_{n=0}^{M-1}\bsone_{\left[0,t\right)}(1-\varphi(n))-t\right) \\
     =&\frac{1}{2}\left(D_{M}^{\varphi}(t)+D_{M}^{1-\varphi}(t)\right).
   \end{align*}
   and for $N=2M+1$ we obtain
   \begin{align*} D_{2M+1}^{\sym}(t)=&\frac{1}{2M+1}\left(\sum_{n=0}^{M}\bsone_{\left[0,t\right)}(\varphi(n))+\sum_{n=0}^{M-1}\bsone_{\left[0,t\right)}(1-\varphi(n))\right)-t \\
     =&\frac{1}{2M+1}\left(\sum_{n=0}^{M}\bsone_{\left[0,t\right)}(\varphi(n))-(M+1)t+\sum_{n=0}^{M-1}\bsone_{\left[0,t\right)}(1-\varphi(n))-Mt\right) \\
     =&\frac{1}{2M+1}\left( (M+1)D_{M+1}^{\varphi}(t)+MD_{M}^{1-\varphi}(t) \right).
   \end{align*}
\end{proof}
\begin{corollary} \label{Cor1}
 We have   
 \[ |\mu_{j,m}^{N,\sym}|\leq \left\{ 
\begin{array}{ll}
\frac{1}{2}\left(|\mu_{j,m}^{M,\varphi}|+|\mu_{j,m}^{M,1-\varphi}|\right) & \mbox{ if } N=2M,\\
\frac{1}{2M+1}\left((M+1)|\mu_{j,m}^{M+1,\varphi}|+M|\mu_{j,m}^{M,1-\varphi}|\right) & \mbox{ if } N=2M+1.
\end{array}\right.
 \]
\end{corollary}

\begin{proof} We consider linearity of integration and the triangle inequality to obtain the result from Lemma~\ref{Diskr}.
\end{proof}

We proceed with the calculation of $\mu_{j,m}$ in the case $j\in \NN_0$ and first prove the following general lemma.

\begin{lemma} \label{allgemein}
  Let $j \in \NN_0$ and $m \in \mathbb{D}_j$. Then for the volume part $f(t)=t$ of the discrepancy function we have
  \[ \mu_{j,m}(f)=-2^{-2j-2} \] and for the counting part $g(t)=\frac{1}{N}\sum_{n=0}^{N-1}\bsone_{\left[0,t\right)}(x_n)$
  we have
  \[ \mu_{j,m}(g)=\frac{2^{-j-1}}{N}\sum_{n=0 \atop x_n \in \stackrel{\circ}{I}_{j,m}}^{N-1}(|2m+1-2^{j+1}x_n|-1), \]
  where $\stackrel{\circ}{I}_{j,m}=\left(\frac{m}{2^j},\frac{m+1}{2^j}\right)$ denotes the interior of $I_{j,m}$.
\end{lemma}

\begin{proof}
   Of course, $$\mu_{j,m}(f)=\int_{0}^{1}t\, h_{j,m}(t)\rd t=\int_{m2^{-j}}^{m2^{-j}+2^{-j-1}}t\rd t-\int_{m2^{-j}+2^{-j-1}}^{(m+1)2^{-j}} t\rd t=-2^{-2j-2}.$$
   The Haar coefficients of $g$ are given by
   \begin{align*}\mu_{j,m}(g)=\int_{0}^{1}\left(\frac{1}{N}\sum_{n=0}^{N-1}\bsone_{\left[0,t\right)}(x_n)h_{j,m}(t)\right)\rd t = \frac{1}{N}\sum_{n=0}^{N-1}\underbrace{\int_{0}^{1}\bsone_{\left[0,t\right)}(x_n)h_{j,m}(t)\rd t}_{\mathcal{I}_n}.
   \end{align*}
   We analyse $\mathcal{I}_n$. If $x_n \notin I_{j,m}$ or $x_n=\frac{m}{2^j}$, it is evident that $\mathcal{I}_n=0$. One can check by simple integration,
   that in the case $x_n\in I_{j,m}^{+}$ we have $\mathcal{I}_n=m2^{-j}-x_n$, and if $x_n\in I_{j,m}^{-}$, then $\mathcal{I}_n=-2^{-j}-(m2^{-j}-x_n)$.
   These results can be combined to
   $$ \mathcal{I}_n=2^{-j-1}(|2m+1-2^{j+1}x_n|-1) \ \ \mbox{ if } \ x_n \in \stackrel{\circ}{I}_{j,m}. $$
   The claimed result follows.
\end{proof}

Now we are ready to show a central lemma.

\begin{lemma} \label{Theo1}
   We have
   $$|\mu_{j,m}^{N,\varphi}|\leq\frac{1}{N}\frac{1}{2^j}\ \ \ \mbox{ and } \ \ \ |\mu_{j,m}^{N,1-\varphi}|\leq\frac{1}{N}\frac{1}{2^j}$$
   for all $0\le j < \lceil\log_{2}{N}\rceil$ and
   $$ |\mu_{j,m}^{N,\varphi}|=|\mu_{j,m}^{N,1-\varphi}|=2^{-2j-2} $$
   for all $j \geq \lceil\log_{2}{N}\rceil$.
\end{lemma}

\begin{proof} 
  We start with $x_n=\varphi(n)$ and investigate the sum
  $$ \sum_{n=0 \atop \varphi(n) \in \stackrel{\circ}{I}_{j,m}}^{N-1}(|2m+1-2^{j+1}\varphi(n)|-1), $$
  which, according to Lemma~\ref{Phi}, we can transfer to 
  \begin{align*}\sum_{s=1}^{A}&\left(\left|2m+1-2^{j+1}\varphi\left(2^j\varphi(m)+2^js\right)\right|-1\right) \\ &=\sum_{s=1}^{A}\left(\left|2m+1-2^{j+1}\left(\frac{m}{2^j}+\varphi(2^js)\right)\right|-1\right) \\
  &=\sum_{s=1}^{A}\left(|1-2^{j+1}\varphi(2^js)|-1\right)=\sum_{s=0}^{A}\left(|1-2\varphi(s)|-1\right).
      \end{align*}
  We still have to specify the upper index. The conditions $0\leq n \le N-1$ and $n=2^j\varphi(m)+2^js$ lead
  to $s \leq \frac{N-1}{2^j}-\varphi(m)$, this is why we choose $A:=\lfloor\frac{N-1}{2^j}-\varphi(m) \rfloor$.
  We also see that there are no elements of $\{x_0,x_1,\dots,x_{N-1}\}$ contained in $\stackrel{\circ}{I}_{j,m}$, if 
  $2^j\varphi(m)+2^j\geq N$, which is fulfilled if $2^j\geq N$ or $j\geq \lceil\log_{2}{N}\rceil$.
  Regarding that fact, we immediately conclude from Lemma~\ref{allgemein} that $|\mu_{j,m}^{N,\varphi}|=2^{-2j-2}$ for $j\geq \lceil\log_{2}{N}\rceil$.
  We proceed with the case $j< \lceil\log_{2}{N}\rceil$ and have
  \begin{align*}
      \mu_{j,m}^{N,\varphi}&=\frac{2^{-j-1}}{N}\sum_{s=0}^{A}\left(|1-2\varphi(s)|-1\right)+2^{-2j-2} \\ &\leq \frac{2^{-j-1}}{N}\left(\frac{A}{2}+1-(A+1)\right)+ 2^{-2j-2}\\
         &=-\frac{2^{-j-2}}{N}A+2^{-2j-2}\leq -\frac{2^{-j-2}}{N}\left(\frac{N-1}{2^j}-\varphi(m)-1\right)+2^{-2j-2} \\
         &=\frac{1}{N}\left(2^{-2j-2}+(\varphi(m)+1)2^{-j-2}\right)\leq \frac{1}{N}\left(2^{-2j-2}+2^{-j-1}\right).
     \end{align*}
     We also find
     \begin{align*}
      \mu_{j,m}^{N,\varphi} &\geq \frac{2^{-j-1}}{N}\left(\frac{A}{2}-(A+1)\right)+ 2^{-2j-2}\\
         &=-\frac{2^{-j-1}}{N}\left(\frac{A}{2}+1\right)+2^{-2j-2}\geq -\frac{2^{-j-1}}{N}\left(\frac{N-1}{2^{j+1}}-\frac{\varphi(m)}{2}+1\right)+2^{-2j-2} \\
         &=\frac{1}{N}\left(2^{-2j-2}+\left(\frac{\varphi(m)}{2}-1\right)2^{-j-1}\right)\geq \frac{1}{N}\left(2^{-2j-2}-2^{-j-1}\right).
     \end{align*}
     By combining these results we finally obtain
     $$|\mu_{j,m}^{N,\varphi}|\leq \frac{1}{N}\left(2^{-2j-2}+2^{-j-1}\right)\leq \frac{1}{N}\frac{1}{2^j}$$
     as claimed.\\
     
     We turn to the estimation of $|\mu_{j,m}^{N,1-\varphi}|$, which can be treated similarly to $|\mu_{j,m}^{N,\varphi}|$.
     To begin with, we observe that
     \begin{align*}
        \sum_{n=0 \atop 1-\varphi(n) \in \stackrel{\circ}{I}_{j,m}}^{N-1}&\left(|2m+1-2^{j+1}(1-\varphi(n))|-1\right) \\
             &=\sum_{n=0 \atop \varphi(n) \in \stackrel{\circ}{I}_{j,2^j-m-1}}^{N-1}\left(|2m+1-2^{j+1}(1-\varphi(n))|-1\right) \\
             &=\sum_{s=1}^{B}\left(\left|2m+1-2^{j+1}\left(1-\varphi\left(2^j\varphi(2^j-m-1)+2^js\right)\right)\right|-1\right) \\
             &=\sum_{s=1}^{B}\left(\left|2m+1-2^{j+1}\left(1-\left(\frac{2^j-m-1}{2^j}+\varphi(2^js)\right)\right)\right|-1\right) \\
             &=\sum_{s=1}^{B}\left(\left|-1+2\varphi(s)\right|-1\right)=\sum_{s=0}^{B}\left(\left|1-2\varphi(s)\right|-1\right).
     \end{align*}
     In this expression, $B:=\lfloor \frac{N-1}{2^j}-\varphi(2^j-m-1) \rfloor$, which we deduce in the same way as the upper index $A$ above. Completely analogously
     as above, we obtain
     $$|\mu_{j,m}^{N,1-\varphi}|\leq \frac{1}{N}\left(2^{-2j-2}+2^{-j-1}\right)\leq \frac{1}{N}\frac{1}{2^j}$$ for $j< \lceil \log_{2}{N}\rceil$.
     The case $j\geq \lceil \log_{2}{N}\rceil$ also follows the same lines as above.
\end{proof}

\begin{corollary} The Haar coefficients of the symmetrized van der Corput sequence for $j \in \NN_0$ fulfil
$$|\mu_{j,m}^{N,\sym}|\left\{ 
\begin{array}{ll}
\leq \frac{1}{2^{j-1}}\frac{1}{N} & \mbox{ if } j < \lceil \log_{2}{N}\rceil,\\ \\
= 2^{-2j-2} & \mbox{ if } j\geq \lceil \log_{2}{N}\rceil.
\end{array}\right.$$
\end{corollary}
\begin{proof} We combine Corollary~\ref{Cor1} and Lemma~\ref{Theo1} to obtain the result. 
\end{proof}

\section{The Proof of Theorem~\ref{thm1}}\label{secproofthm}

We are ready to show that the $L_p$-discrepancy of the symmetrized van der Corput sequence has optimal order in $N$ for any $p \in (1,\infty)$. We apply the {\it Littlewood-Paley inequality} which involves the {\em square function} 

$$ S(f) := \left( \sum_{j \in \NN_{-1}} \sum_{m \in \mathbb{D}_{j}} 2^{2\max\{0,j\}} \, \langle f , h_{j,m} \rangle^2 \, {\mathbf 1}_{I_{j,m}} \right)^{1/2}$$  
of a function $f\in L_p([0,1))$. 

\begin{lemma}[Littlewood-Paley inequality]\label{lpi}
Let $p \in (1,\infty)$. Then there exist $c_p,C_p>0$ such that for every $f\in L_p([0,1))$ we have 
 $$c_p \| f \|_{L_p} \le \| S(f) \|_{L_p} \le C_p \| f \|_{L_p}.$$
\end{lemma}

Proofs of these inequalities and further details also yielding the right asymptotic behavior of the involved constants can be found in \cite{bur88,ste93,WG91}. Littlewood-Paley theory has already been used in the context of discrepancy before, see, e.g., \cite{DHP,HKP14,Skr,Skr12}.

Now we can give the proof of Theorem~\ref{thm1}.

 \begin{proof} Throughout this proof, we simply write $\mu_{j,m}$ instead of $\mu_{j,m}^{N,\sym}$. Using Lemma~\ref{lpi} with $f=D_{N}(\cH^{\sym},\cdot)$ we have
\begin{eqnarray*}
L_{p,N}(\cH^{\sym}) & = & \|D_{N}(\cH^{\sym},\cdot)\|_{L_p}\\
& \ll_p & \|S(D_{N}(\cH^{\sym},\cdot))\|_{L_p}
\\
& = & \left\| \left( \sum_{j \in \NN_{-1}} \sum_{m \in \mathbb{D}_{j}} 2^{2\max\{0,j\}} \, \mu_{j,m}^2 \, {\mathbf 1}_{I_{j,m}} \right)^{1/2} \right\|_{L_p}\\
& = & \left\|\sum_{j \in \NN_{-1}} 2^{2\max\{0,j\}} \sum_{m \in \mathbb{D}_{j}} \, \mu_{j,m}^2 \, {\mathbf 1}_{I_{j,m}} \right\|_{L_{p/2}}^{1/2}\\
& \le & \left( \sum_{j \in \NN_{-1}} 2^{2\max\{0,j\}} \left\| \sum_{m \in \mathbb{D}_{j}} \, \mu_{j,m}^2 \, {\mathbf 1}_{I_{j,m}} \right\|_{L_{p/2}} \right)^{1/2},
\end{eqnarray*}
where we used Minkowski's inequality for the $L_{p/2}$-norm. Hence, in order to prove the result it suffices to show that 
\begin{equation*}\label{eq2do}
\sum_{j \in \NN_{-1}} 2^{2\max\{0,j\}} \left\| \sum_{m \in \mathbb{D}_{j}}  \mu_{j,m}^2 \, {\mathbf 1}_{I_{j,m}} \right\|_{L_{p/2}} \ll \frac{\log{N}}{N^2}.
\end{equation*}
Now Lemma~\ref{erster} and Lemma~\ref{Theo1} give
\begin{align*}
\sum_{j \in \NN_{-1}} 2^{2\max\{0,j\}} \left\| \sum_{m \in \mathbb{D}_{j}}  \mu_{j,m}^2 \, {\mathbf 1}_{I_{j,m}} \right\|_{L_{p/2}} \le & \frac{1}{4N^2} \| \bsone_{\left[0,1\right)} \|_{L_{p/2}}  
        +\frac{4}{N^2}\sum_{j=0}^{\lceil \log_{2}{N} \rceil -1}\left\| \sum_{m \in \mathbb{D}_{j}}  \, {\mathbf 1}_{I_{j,m}} \right\|_{L_{p/2}} \\ 
       &+\frac{1}{16}\sum_{j=\lceil \log_{2}{N} \rceil}^{\infty}2^{-2j}\left\| \sum_{m \in \mathbb{D}_{j}}   \, {\mathbf 1}_{I_{j,m}} \right\|_{L_{p/2}} \\ 
       \le & \frac{1}{4N^2}+\frac{4}{N^2}(\log_{2}{N}+1)+\frac{1}{12}\frac{1}{4^{\lceil \log_{2}{N} \rceil}} \ll \frac{\log{N}}{N^2}
\end{align*}
where we regarded the fact that $\sum_{m \in \mathbb{D}_{j}}  \, {\mathbf 1}_{I_{j,m}}=1$ for a fixed $j \in \NN_0$. The proof is complete. 
 \end{proof}

\noindent{\bf Author's Addresses:}

\noindent Ralph Kritzinger and Friedrich Pillichshammer, Institut f\"{u}r Finanzmathematik, Johannes Kepler Universit\"{a}t Linz, Altenbergerstra{\ss}e 69, A-4040 Linz, Austria. Email: ralph.kritzinger(at)jku.at, friedrich.pillichshammer(at)jku.at

\end{document}